\documentclass[12pt]{amsart}
\usepackage{amsfonts,amssymb,latexsym, amsmath, amsthm, amsrefs,enumitem,verbatim}

\newtheorem{thm}{Theorem}[section]
\newtheorem{lem}[thm]{Lemma}
\newtheorem{cor}[thm]{Corollary}
\newtheorem{prop}[thm]{Proposition}
\newtheorem{conj}{Conjecture}

\newtheorem{thmintro}{Theorem}

\theoremstyle{definition}

\newtheorem{rem}[thm]{Remark}

\newcommand{\eps}{\varepsilon}

\newcommand{\SL}{{\mathrm {SL}}}

\newcommand{\CU}{{\mathcal U}}

\newcommand{\C}{{\mathrm C}}

\def\geq{\geqslant}

\newcommand{\Dih}{\mathrm{Dih}}

\def\Z{{\mathbb Z}}
\def\N{{\mathbb N}}
\def\QQ{{\mathbb Q}}
\def\FF{{\mathbb F}}

\def\NN{{\mathrm N}}
\def\R{{\mathrm R}}

\def\pim{{\pi'_{\rm min}}}

\def\leq{\leqslant}
\def\geq{\geqslant}
\def\cont{\mathfrak{c}}

\usepackage[usenames]{color}

\begin{document}

\title{First-order recognisability in finite and pseudofinite groups} 

\author{Yves Cornulier and John S. Wilson} 

\address{Y. Cornulier: CNRS and Univ Lyon, Univ Claude Bernard Lyon 1, Institut Camille Jordan, 43 blvd du 11 novembre 1918, F-69622 Villeurbanne}
\address{J.S. Wilson: Christ's College, Cambridge CB2 3BU, Great Britain\\and Mathematisches Institut, Universit\"at Leipzig,
Germany}

\date{February 17, 2020}



\begin{abstract}
It is known that there exists a first-order sentence that holds in a finite group if and only if the group is soluble. Here it is shown that the corresponding statements with `solubility' replaced by `nilpotence' and `perfectness', among others, are false.

These facts present difficulties for the study of pseudofinite groups. However, a very weak form of Frattini's theorem on the nilpotence of the Frattini subgroup of a finite group is proved for pseudofinite groups.
\end{abstract}
\maketitle 

\section{Introduction}
We are concerned with properties of groups that can, for finite groups, be expressed by a sentence in the first-order language of group theory. Here we view a group as a set endowed with a binary operation (group multiplication), a self-map (the inverse map) and a constant (the element $1$) satisfying the usual axioms. Recall that a sentence means a formula with no free variable.

For instance, for each natural number $n$, a finite group satisfies the first-order universal sentence $(\forall x) (x^n=1\to x=1)$ if and only if it has order coprime to $n$. Commutativity, or more generally nilpotence
of class at most $n$, can similarly be characterized by a universal sentence .  The first-order sentence $(\forall x)\,((x=1)\vee (\exists y)([x,x^y]\neq 1))$ characterizes groups with no non-trivial soluble normal subgroup.

More generally, we consider first-order formulae with free variables. Given a first-order formula with $n+1$ free variables $F(\underline x,y)=F(x_1,\dots,x_n,y)$, and given a group $G$ with an $n$-tuple $\underline g=(g_1,\dots,g_n)$, by $F(\underline g,G)$ we mean the subset $\{h\in G\mid F(\underline g,h)\}$. A subset of this form is called a definable subset; if $n=0$ we call it a parameter-free definable subset and we call the formula $F(g)$ parameter-free. 
For instance, the formula $(\forall y)(xy=yx)$ defines the centre of a group; more generally, for each $k$ there exists a parameter-free first-order formula defining the $k$th term of the central series in every group.

Any group property that can be characterized by a first-order sentence is stable under taking ultraproducts.  This makes it easy to show that many group properties, such as finiteness, simplicity, perfectness, solubility and nilpotence, are not characterizable by a first-order formula. Nevertheless, we can ask whether such properties can be characterized {\it among finite groups} by a first-order sentence, that is, whether there is a first-order sentence that is satisfied by a finite group $G$ if and only if $G$ has the property. 

It was shown in \cite{Wsol} that a finite group $G$ is soluble if and only if it satisfies the first-order sentence $\sigma_{56}$ expressing the fact that no non-trivial element $g$ is a product of 56 commutators $[x,y]$ with both of $x,y$ conjugate to $g$, and moreover in \cite{Wrad} that there is a parameter-free first-order formula $\rho(x)$\label{rho_def} such that the set $\rho(G)=\{x\mid \rho(x)\}$ is equal to the soluble radical $\R(G)$ of $G$.  We recall that the soluble radical, Fitting subgroup and Frattini subgroup of a finite group $G$ are respectively the largest soluble normal subgroup, the largest nilpotent normal subgroup and the intersection of all maximal subgroups.  

A finite group $G$ is called $\varpi$-group, for a set $\varpi$ of primes, if every prime divisor of $|G|$ is in $\varpi$. By $\varpi'$, we mean the complement of $\varpi$ in the set of primes. Thus, a finite group 
is a $\varpi'$-group if and only if its order is divisible by no prime in $\varpi$. For $\varpi$ finite, the 
$\varpi'$-groups are characterized, among finite groups, by the sentence $(\forall x) (x^n=1\to x=1)$ seen above, with $n$ the product of the primes in $\varpi$, and also by the sentence $(\forall x)(\exists y)(x=y^n)$. By contrast, we have assertion \ref{item_ifS} below.

\begin{thmintro}\label{thm_abel}~
\begin{enumerate}[label={\rm (\alph*)}]
\item\label{item_ifS} 
Let $\varpi$ be a set of primes. If there is a first-order sentence that holds for a finite cyclic group $C$ if and only if it is a $\varpi'$-group, then $\varpi$ is either finite or the set of all primes. 

In particular, there is no first-order sentence that holds for a cyclic group $C$ if and only if $|C|$ is a power of $2$.
\item\label{item_nil} There is no first-order sentence that holds for a finite group $G$ if and only if $G$ is nilpotent. 

Indeed, if $\varpi$ is a set of primes, then there is a sentence that characterizes the nilpotent groups among the
finite $\varpi$-groups if and only if $\varpi$ is finite. 
\item\label{item_Fit} There is no first-order formula $\chi(x)$ such that for every finite group $G$ the Fitting subgroup is equal to $\chi(G)$. 
\end{enumerate}
\end{thmintro}

Assertions \ref{item_nil}, \ref{item_Fit} answer negatively questions that have been raised from time to time in the literature (e.g.\ in Question 3.0.11 in \cite{Macp}).   

It is convenient to use the following language.  We say that two sequences $(G_n)$, $(H_n)$ of groups are {\em asymptotically elementarily equivalent} (AEE) if for every first-order sentence $F$, there exists $m=m_F$ such that for every $n\geq m$, the group $G_n$ satisfies $F$ if and only if the group $H_n$ satisfies $F$. 

Let $\mathcal{C}$ be an isomorphism-closed class of groups.
Then, for a pair $((G_n),(H_n))$ of sequences of finite groups, we have the following assertions (see \S\ref{s_preli}):
\begin{itemize}
\item if for every non-principal ultrafilter $\CU$, the ultraproducts $\prod_\CU G_n$ and $\prod_\CU H_n$ are isomorphic, then $((G_n),(H_n))$ is AEE;
\item if  $((G_n),(H_n))$ is AEE and 
$G_n\in\mathcal{C}$, $H_n\notin\mathcal{C}$ for all $n$, then there is no formula $\Phi$ such that a finite group $G$ is in $\mathcal{C}$ if and only if it satisfies $\Phi$.
\end{itemize}

Some of the statements of Theorem \ref{thm_abel} will be obtained by exhibiting suitable AEE sequences of finite groups; we prove that the sequences are AEE by showing that they have isomorphic ultraproducts.
For Theorem \ref{thm_abel}\ref{item_ifS}, we use a pair of sequences $((G_n),(H_n))$ satisfying $H_n=G_n\times C_{p_n}$, where $(p_n)$ is a sequence of primes tending to infinity.  The ultraproduct of the second sequence is the direct product of the ultraproduct $U$ of the first sequence and the ultraproduct of the groups $C_{p_n}$. It is easy to see that the latter ultraproduct is isomorphic to $\QQ^{(\cont)}$, the unique torsion-free divisible abelian group up to isomorphism with the cardinality $\cont$ of the continuum.  If we can show that 
$U$ has a direct factor isomorphic to $\QQ^{(\cont)}$, it then follows that the ultraproducts of the sequences $(G_n)$ and $(G_n\times C_{p_n})$ are isomorphic.  For example, we show (Lemma \ref{summand_R}) that if  $(A_n)$ is a sequence of finite abelian groups of exponents tending to infinity then every ultraproduct of the sequence $(A_n)$ has a direct summand isomorphic to $\QQ^{(\cont)}$  and this leads easily to assertion \ref{item_ifS} of Theorem A.

To deduce (the first part of) assertion \ref{item_nil}, we use the fact that taking ultraproducts commutes with the operation $\Dih$ which maps an abelian group $A$ to the semidirect product $A\rtimes_\pm C_2$ in which the non-trivial element of $C_2$ acts by inversion.

Next we prove, with more work, the existence of a summand $\QQ^{(\cont)}$ in suitable ultraproducts of finite perfect groups $(G_n)$. As above, it follows that $(G_n)$ and $(G_n\times C_{p_n})$ are AEE sequences, yielding: 

\begin{thmintro}\label{thm_perfect}
There is no first-order sentence that holds for a finite group $G$ if and only if $G$ is perfect. 
\end{thmintro}

This strengthens the assertion (cf.\  \cite[Lemma 2.1.10]{holt}) that there exist finite perfect groups of arbitrarily large commutator widths. The latter result is {\em a priori} weaker, since it only guarantees that every first-order sentence satisfied by all finite perfect groups is also satisfied by some non-perfect group $H$; Theorem \ref{thm_perfect} asserts that $H$ can be chosen to be finite.

Felgner established in \cite{felgner} that among finite groups, the non-abelian simple groups can be characterized by a first-order sentence.  Moreover the class of finite direct products of non-abelian simple groups can also be characterized by a single sentence, as follows easily from \cite{Wcomp} (see Proposition \ref{form_produ}).  Among such direct products, for each given $n\geq 1$ the property of being a direct product of at least $n$ simple groups is obviously characterizable by the formula
\[(\exists x_1\dots\exists x_n)(\forall y)\Big(\bigwedge_i x_1\neq 1\wedge \bigwedge_{1\leq i<j\leq n} [x_i,yx_jy^{-1}]=1\Big).\]
However, among these direct products, direct powers of simple groups cannot be recognized:

\begin{thmintro}\label{thm_power}
There is no first-order sentence that holds for a finite group $G$ that holds if and only if it is a direct power $($resp.\ direct square$)$ of a non-abelian simple group.
\end{thmintro}

The proof again uses ultraproducts and the fact that taking ultraproducts commutes with taking direct products of two groups. Perhaps surprisingly, in our proof we use the continuum hypothesis (CH): this  ensures that an ultraproduct of finite groups is determined up to isomorphism by its elementary theory. Then Shoenfield's absoluteness theorem is used to eliminate CH from the assumptions.   

Given a class $\mathcal{C}$ of groups, following \cite{OHP} we say that a group is pseudo-$\mathcal{C}$ if it satisfies all first-order sentences satisfied by all groups in the class $\mathcal{C}$. Basic arguments (cf.\ Proposition \ref{pseudo_C}) show the equivalence the following properties:  
\begin{enumerate}[label={\rm (\roman*)}]
\item $G$ is pseudo-$\mathcal{C}$;
\item $G$ is elementarily equivalent to some ultraproduct of a sequence of groups in $\mathcal{C}$;
\item there exists a sequence $(G_n)$ of groups in $\mathcal{C}$ that elementarily converges to $G$ (that is, the the sequence $(G,G_n)$ is an AEE pair in the above sense). 
\end{enumerate}

In particular, a group is called {\em pseudofinite} if it satisfies all first-order sentences in the language of group theory that hold in all finite groups.  (Some authors require pseudofinite groups to be infinite;  here we find it preferable to include finite groups.)  
 
Some results about finite groups extend directly to pseudofinite groups.  For example, a group is pseudo-(finite soluble) if and only if it is pseudofinite and satisfies the sentence $\sigma_{56}$ mentioned above. Since for arbitrary groups $\neg\sigma_{56}$ is clearly an obstruction to solublility, it follows that the pseudo-(finite soluble) groups are just the groups that are pseudofinite and pseudosoluble.

However the negative assertions of Theorem \ref{thm_abel} present obstacles the study of pseudofinite groups. Indeed, the notion of being pseudo-(finite nilpotent) is worse-behaved than its soluble analogue. 
Theorem \ref{thm_abel}\ref{item_nil} shows that a pseudo-(finite nilpotent) group can also be pseudo-(finite non-nilpotent). That is, among pseudofinite groups, pseudo-(finite nilpotent) groups cannot be characterized by a single formula. We note that pseudonilpotent groups satisfy various properties not true in general soluble finite groups, such as $$(\forall x\forall y)\, (x^2=1\wedge y^3=1\to xy=yx).$$ Elaborating on the construction  of Theorem \ref{thm_abel}\ref{item_nil}, we obtain the following pathology: 

\begin{thmintro}\label{psfn_subgroups}
There exists a pseudofinite group $G$ having $($normal\/$)$ definable subgroups $H\leq L$ with $L$ pseudo-(finite nilpotent) but not $H$. 
\end{thmintro}

In our example, $H\neq1$ and $H$ has trivial centre so is not pseudonilpotent;  $L$ has index $8$ and is elementarily equivalent to an ultraproduct of dihedral 2-groups; $H$ is definable in $G$ but not in $L$.   Indeed, a definable subgroup of a pseudo-(finite nilpotent) group is pseudo-(finite nilpotent), by a routine argument: more generally, if a class $\mathcal{C}$ of groups is closed under taking subgroups, then being pseudo-$\mathcal{C}$ is closed under taking definable subgroups.

Clearly a pseudo-(finite nilpotent) group is both pseudofinite and pseudonilpotent. We do not know whether the converse implication holds. 

A definable subgroup of a pseudofinite group is also pseudofinite. Moreover, in a pseudofinite group $G$ with a pseudosoluble definable subgroup $H$, all definable subgroups of $G$ contained in $H$ are pseudosoluble (although they can fail to be definable in $H$).

Let $G$ be a pseudofinite group. Define $\R(G)=\{g\in G\mid \rho(g)\}$, where $\rho$ is defined at the beginning of this introduction.  It is parameter-free definable, and contains all pseudosoluble definable normal subgroups of $G$.
Obviously $\R(G)=1$ if and only if $G$ has no non-trivial abelian normal subgroup. Furthermore, in general the quotient $G/\R(G)$ is pseudofinite and $\R(G/\R(G))=1$ (see for example \cite[Propositions 2.16, 2.17]{OHP}).  

By \cite[Theorem 1]{Wcomp}, there are first-order formulae $\pi(h,x)$, $\pim(h)$ such that the non-abelian minimal normal subgroups of a finite group $G$ are precisely the sets $\pi(h,G)$  for elements
$h\in G$ satisfying $\pim(h)$.  From this, it was shown in \cite{Wcomp} that every nontrivial pseudofinite group with $\R(G)=\{1\}$ has a minimal  (non-trivial) definable normal subgroup, and indeed that every nontrivial definable normal subgroup contains a minimal one.

A basic result of finite group theory is that in every finite group, every non-nilpotent normal subgroup admits a proper supplement.
 (A proper supplement to $K\triangleleft G$ is a proper subgroup $H<G$ such that $G=HK$.) This is essentially equivalent to the fact that the Frattini subgroup is nilpotent in every finite group. It is not clear at this point to which extent this can be generalized to the setting of pseudofinite groups. In this direction, we obtain the following partial generalization.

\begin{thmintro}\label{thm_defina}
 Let $G$ be a pseudofinite group and $K$ a normal definable subgroup. If $K$ is not pseudosoluble, then it admits a proper definable supplement in $G$.

 Consequently, every definable subgroup of $G$ with no proper supplement is contained in~$\R(G)$. 
\end{thmintro}

We do not know whether it is true that if $G$ is pseudofinite and $N$ is a normal definable subgroup that is not pseudo-(finite nilpotent)\footnote{$N$ not pseudo-(finite nilpotent) {\it inside $G$} would be a more natural assumption; see Remark \ref{nilpotent_inside}.}, then $N$ admits a proper definable supplement. For instance, if $N$ has two non-commuting elements of finite coprime order, does it have a proper definable supplement?

Notice that there is no first-order formula $f(x)$ defining the Frattini subgroup in every finite group.
If so, the first-order sentence $(\forall x)(f(x)\to x=1)$ would then characterize the finite groups with trivial Frattini subgroup, but no such sentence exists. Indeed, take a sequence of primes $(p_n)$ tending to infinity. Then $C_{p_n}$ has a trivial Frattini subgroup, but not $C_{p_n^2}$, but $(C_{p_n})$ and $(C_{p_n^2})$ are AEE sequences: indeed all non-principal ultraproducts of these families of groups are isomorphic to~$\QQ^{(\cont)}$.

\bigskip

\noindent {\bf Acknowledgements.}  The work reported here was begun at the conference 'Groups and their Actions' in Gliwice, Poland, in September 2019, and both authors are grateful to the local organizers for providing ideal circumstances for mathematical discussions. The first-named author thanks Todor Tsankov for his kind explanation of Shoenfield's absoluteness theorem. 
 
\section{Preliminaries}\label{s_preli}

The results of this section are classical and valid in a wider setting, but stated in a group-theoretic form fashioned to our needs.
We shall consider ultraproducts of finite groups: we recall that the ultraproduct $\prod_\CU G_i$ of groups $G_i$ with $i\in \N$ with respect to a ultrafilter $\CU$ on $\N$ is the quotient of the unrestricted direct product $\prod_iG_i$ by the normal subgroup $\big\{(f_i)_{i\in\N}\in \prod_\CU G_i\mid\{i: f_i=1\}\in\CU\big\}$. It has the property (guaranteed by \L os's theorem) that a first-order sentence $\theta$ holds in it if and only if the set of $i$ such that $\theta$ holds in $G_i$ is in $\CU$.  For general facts about ultraproducts and \L os's theorem we refer the reader to \cite{eklof}.

  We need the following basic observations:

\begin{prop}\label{AEE_car}
Let $\mathcal{C}$ be an isomorphism-closed class of groups. The following are equivalent:
\begin{enumerate}[label={\rm (\roman*)}]
\item\label{item_AEE1}
there exists a first-order formula $F$ such that a finite group is in $\mathcal{C}$ if and only if it satisfies $F;$
\item\label{item_AEE2}
there is no AEE pair $((G_n),(H_n))$ of sequences of finite groups, with $G_n\in\mathcal{C}$, $H_n\notin\mathcal{C}$ for all $n$.
\end{enumerate}
\end{prop}
\begin{proof}
If $(G_n)$, $(H_n)$ are AEE sequences with $G_n\in\mathcal{C}$, $H_n\notin\mathcal{C}$ for all $n$, then any formula $f$ that holds 
for all finite groups in $ \mathcal{C}$ holds for all $G_n$, so holds in $H_n$ for large $n$. Therefore \ref{item_AEE1} implies \ref{item_AEE2}.  

Conversely, suppose that \ref{item_AEE1} fails. The set of first-order sentences (with variables drawn from a fixed countable set) is countable; let $(f_n)$ be an enumeration of its elements and write $g_n=\bigwedge_{k\leq n} f_k$ for each $n$.  Since $g_n$ does not characterize the family of finite groups in $\mathcal{C}$, there are finite groups $G_n$, $H_n$ in which $g_n$ holds with $G_n\in \mathcal{C}$ and $H_n\notin \mathcal{C}$.  Evidently $((G_n),(H_n))$ is an AEE pair of sequences and  the result follows.
\end{proof}

\begin{prop}\label{AEE_ultra}
Let $(G_n)$, $(H_n)$ be two sequences of groups. Then $(G_n)$ and $(H_n)$ are AEE if and only if for every non-principal ultrafilter $\CU$ on the integers, the ultraproducts ${\prod}_\CU G_n$ and ${\prod}_\CU H_n$ are elementarily equivalent.

In particular, if for every $\CU$, the groups ${\prod}_\CU G_n$ and ${\prod}_\CU H_n$ are isomorphic, then $(G_n)$ and $(H_n)$ are AEE.
\end{prop}
\begin{proof}
If there exist a first-order sentence $F$ and an infinite subset $I$ such that for every $n\in I$, the group $G_n$ satisfies $F$ but not $H_n$, choose a non-principal ultrafilter $\CU$ supported by $U$. Then $F$ holds in ${\prod}_\CU G_n$ but not in ${\prod}_\CU H_n$. The converse is immediate too.
\end{proof}

\begin{rem}\label{CH_rem}
If we assume the continuum hypothesis and $(G_n)$, $(H_n)$ are sequences of countable groups, then 
${\prod}_\CU G_n$ and ${\prod}_\CU H_n$ are elementarily equivalent if and only if they are isomorphic. This is because they are $\omega_1$-saturated structures of cardinality $\omega_1$ and hence determined up to isomorphism by their first-order theories (see \cite[Theorem 9.7]{Poi}). For the same reason, under the continuum hypothesis if $(G_n)$ elementarily converges to some group, then all non-principal ultraproducts ${\prod}_\CU G_n$ are isomorphic.
\end{rem}

\begin{prop}\label{pseudo_C}
Let $\mathcal{C}$ be a class of groups and let $G$ be a group. The following are equivalent:
\begin{enumerate}[label={\rm (\roman*)}]
\item\label{item_pc} $G$ is pseudo-$\mathcal{C}$;
\item\label{item_eeu} $G$ is elementarily equivalent to some ultraproduct of a sequence of groups in~$\mathcal{C}$;
\item\label{item_aee} there exists a sequence $(G_n)$ of groups in $\mathcal{C}$ that elementarily converges to $G$ $($that is, the the sequence $(G,G_n)$ is an AEE pair in the above sense$)$. 
\end{enumerate}
\end{prop}
\begin{proof}
The implications \ref{item_aee}$\Rightarrow$\ref{item_eeu}$\Rightarrow$\ref{item_pc} hold by \L os's theorem. Suppose that \ref{item_pc} holds. Let $(f_n)$ be an enumeration of the first-order sentences holding in $G$ (with variables in a given countable set) and define $g_n=\bigwedge_{k\leq n}f_k$ for each $n$. For each $n$,  since $G$ is pseudo-$\mathcal C$ and satisfies $g_n$
there is a group $G_n$ in $\mathcal C$ satisfying $g_n$.  Then the sequence $(G_n)$ converges to $G$.
\end{proof}

\section{Abelian groups, nilpotence and Theorem \ref{thm_abel}}

\begin{prop}\label{elab}  Let $E$ be an infinite elementary abelian $p$-group for a prime $p$.  Then the definable subsets of $E$ are precisely the finite and cofinite subsets.

In particular, a proper subgroup of $E$ is definable if and only if is finite; in particular there is no maximal definable subgroup.
\end{prop}

\begin{proof}  Clearly finite and cofinite subsets are definable.  Now let $L$ be a definable subset; thus there exist a $k$-tuple $(e_1,\dots,e_k)$ of elements of $E$ and a formula $P(y_1,\dots, y_k, x)$ with $L=P(e_1,\dots,e_k,E)$. 
Write $F=\{e_1,\dots,e_k\}$ and let $A$ be the group of automorphisms of $E$ fixing $F$ pointwise.  Then $A$ acts transitively on the set
$E\smallsetminus\langle F\rangle$, and since $A$ must map $L$ to itself it follows that $L\subseteq \langle F\rangle$ or $E\smallsetminus \langle F\rangle \subseteq L$.
\end{proof}

The next two results give one of the implications in the second assertion of Theorem \ref{thm_abel}(b).

\begin{lem}\label{nilppi}
Let $G$ be a finite group. For integers $p,q$, let $F_{p,q}$ be the sentence 
\[(\forall x\forall y)\big(([x^p,y]=1\wedge [x,y^q]=1)\to [x,y]= 1\big).\]
The following are equivalent:
\begin{enumerate}[label={\rm (\roman*)}]
\item\label{item_gni} $G$ is nilpotent;
\item\label{item_pqf} for all distinct primes $p,q$, $G$ satisfies $F_{p,q}$;
\item\label{item_pqfd} for all distinct primes dividing $|G|$, $G$ satisfies $F_{p,q}$.
\end{enumerate}
\end{lem}
\begin{proof}
\ref{item_gni}$\Rightarrow$\ref{item_pqf} Let $G$ be a finite nilpotent group, let $p,q$ be distinct primes and let $x,y\in G$ satisfy $[x^p,y]=[x,y^q]=1$. Write $G$ as direct product of its Sylow subgroups and for each prime $s$ denote by $x_s$ the projection of $x$ in the Sylow $s$-subgroup. If $s\neq p$ then $[x_s^p,y_s]=1$ and hence $[x_s,y_s]=1$ (since $x_s$ is a power of $x_s^p$). Similarly, $[x_s^p,y_s]=1$ for $s\neq q$. Hence for every $s$ we have $[x_s,y_s]=1$; thus $[x,y]=1$.

\ref{item_pqf}$\Rightarrow$\ref{item_pqfd} is trivial.

\ref{item_pqfd}$\Rightarrow$\ref{item_gni}  Suppose that $G$ is non-nilpotent. Then there exist distinct prime divisors $p,q$ of $|G|$ and non-commuting elements $x,y$ with $x$ of $p$-power order and $q$ of $q$-power order. We can replace $x$ by $x^p$ and $y$ by $y^q$ as many times as necessary to ensure that $[x,y^q]=[x^p,y]=1\neq [x,y]$.
\end{proof}

\begin{cor}\label{nilppi2}
Among $\varpi$-groups with $\varpi$ finite, nilpotence is characterized by the universal formula $\bigwedge_{p\neq q\in\varpi}F_{p,q}$.\qed
\end{cor}

Let us now proceed to the ultraproduct constructions needed for Theorem \ref{thm_abel}.

For each $n$ write $C_n$ and $D_{2n}$ respectively for a cyclic group of order $n$ and a dihedral group of order $2n$. We recall that $\cont=2^{\aleph_0}$ denotes the cardinality of the continuum.  We also
  recall that a group is divisible if the power map $g\mapsto g^n$ is surjective for every integer $n\geq 1$.

We use freely the basic facts that if $A$ is an (additive) abelian group, then every divisible subgroup is a direct summand, and if $A$ is torsion-free then its largest divisible subgroup is $\bigcap n!A$.

\begin{lem}\label{summand_R} Let $(A_n)$ be a sequence of abelian groups, and let $\eps_n\in\N_{\geq 1}\cup\{\infty\}$ be the exponent of $A_n$. Let $\CU$ be a non-principal ultrafilter $\CU$ on $\N$. Then the ultraproduct $A=\prod_\CU A_n$ satisfies $A\cong A\times\QQ^{(\cont)}$ if and only if $\lim_\CU\eps_n=\infty$.
\end{lem}

\begin{proof} 
The condition $\lim_\CU\eps_n=\infty$ is clearly necessary, since otherwise $A$ has finite exponent. Hence suppose that $\lim_\CU\eps_n=\infty$; we have to prove that $A$ contains a copy of $\QQ^{(\cont)}$ (which is then automatically a direct summand).   If there are subgroups $B_i\leq A_i$ such that $\prod_\CU B_i$ has a subgroup isomorphic to $\QQ^{(\cont)}$ then so has the larger group $A$, and so we may replace each $A_i$ by a cyclic subgroup of sufficiently large order and assume that each $A_i$ is cyclic.  
We start with two particular cases: the first of them is in fact enough for most of our purposes.

Suppose that for some fixed prime $p$, each $A_n$ is cyclic of order some power $p^{k_n}$ of $p$.  Let $P=\mathrm {Cr}\, A_n$ and write bars for images of subgroups of $P$ in $A$.  Write $l_n=\lfloor k_n/2\rfloor$,  and for each integer $r\geq0$ let $L_r$ be the Cartesian product of the groups 
$p^{\max(0, l_n-r)}\, A_n$.   
Plainly for each $r$ we have $qL_r=L_r$ for all primes $q\neq p$. Since $L_{r-1}$ is the sum of $pL_r$ and a subgroup of elements with only finitely many non-zero entries,  we have $p\overline {L_r}=\overline {L_{r-1}}$. Hence $(\bigcup \overline {L_r})/\overline {L_0}$ is a divisible group and so $n!A\geq (\bigcup \overline {L_r})\geq \overline {L_0}$ for each $n\geq1$.  The group $\overline {L_0}$ is the ultraproduct of the groups $p^{l_n} A_n$ and so it has cardinality $\cont$.
Since the torsion subgroup $T$ of $A$ has $(p-1)p^{r-1}$ elements of order $p^r$ for all
$r$ it is a Pr\"ufer group of rank $1$.  Therefore $\bigcap n!(A/T)$ is a divisible torsion-free group of cardinality $\cont$ and hence isomorphic to $\QQ^{(\cont)}$. Since $A$ splits over the (divisible) subgroup $T$, the result now follows in this case.

Next suppose that all subgroups $A_i$ are infinite cyclic.  Then $A$ is torsion-free, and so its largest
divisible subgroup is $\bigcap_{n>0} n!A$.  However the ultraproduct $\prod_\CU n!A$ lies in the intersection, and hence this
intersection has cardinality $\cont$.

We use now the elementary fact that if $\N$ is the disjoint union of 
finitely many subsets $I_1, \dots ,I_r$ then 
$I_j \in \CU$ for some $j$, and then
${\mathcal V}=\{S\cap I_j \mid S\in \CU\}$ 
is an ultrafilter on $I_j$, and moreover $A$ is isomorphic to the ultraproduct of the groups $\{G_i\mid i\in I_j\}$ with respect to the ultrafilter $\mathcal V$.   

Because of this  and what was proved above, it remains just to consider the case in which all groups $A_i$ are finite, and for each prime $p$ the set $\{i\mid p{\not|} \;|A_i| \}$
is in $\CU$.   But in this case $\prod_\CU A_n$ is clearly divisible and torsion-free, so isomorphic to $\QQ^{(\cont)}$.
\end{proof}

\begin{rem}
{\rm It follows from Szmielew's \cite{Szm} classification of abelian groups up to elementary equivalence that every non-principal ultraproduct of the sequence $(C_{p^n})$ is elementarily equivalent to $C_{p^\infty}\times\mathbb{Z}_{(p)}$, and more generally to $C_{p^\infty}\times\mathbb{Z}_{(p)}\times\QQ^{(\alpha)}$ for every cardinal $\alpha$. Here $\Z_{(p)}=\Z_p\cap\QQ$ is the additive group of rationals with denominator coprime to the prime~$p$.}
\end{rem}

\begin{lem}  Let $(p_n)$ be a sequence of distinct odd primes.  Then 
$$({\mathrm a}) \;\; {\prod}_\CU C_{2^np_n} \cong {\prod}_\CU C_{2^n}.\qquad\hbox{and}\qquad  ({\mathrm b})\;\; {\prod}_\CU D_{2^np_n}\cong {\prod}_\CU D_{2^n}.$$
\end{lem}

\begin{proof}  (a) We have $\prod_\CU C_{2^np_n}  \cong \prod_\CU C_{2^n} \oplus \prod_\CU C_{p_n}$.  From Lemma \ref{summand_R}, the group $\prod_\CU C_{2^n}$ has the form $H\oplus \QQ^{(\cont)}$ for a subgroup $H$.  The group $\prod_\CU C_{p_n} $ is divisible and torsion-free (for example by \L os's theorem) and so is isomorphic to $\QQ^{(\cont)}$.  Hence
$$  {\prod}_\CU C_{2^np_n} \cong (H\oplus \QQ^{(\cont)})\oplus \QQ^{(\cont)}\cong H\oplus \QQ^{(\cont)}\cong {\prod}_\CU C_{2^n}.$$

(b)  In each term, the ultraproduct $G$ is a split extension of the corresponding ultraproduct $A$ in (a) by a subgroup $\langle t\rangle$ of order $2$; the element $t$ is the image in $G$ of an element $\hat t$ 
of the (unrestricted) direct product containing a non-central involution in each entry. The element $\hat t$ acts as inversion on the direct products of cyclic groups and hence $t$ acts as inversion on $A$; in particular every subgroup of $A$ is $t$-invariant. The isomorphism in (a) preserves the action of $t$. The result follows.
\end{proof} 

\begin{proof}[Proof of Theorem \ref{thm_abel}]
 If $\theta$ is a first-order sentence that holds in all groups $C_{2^n}$ then it holds for $\prod_\CU{C_{2^n}}$, and so from above it holds in all but finitely many groups $C_{2^np_n}$.
Similarly, a sentence that holds in all nilpotent groups also holds in $\prod_\CU{D_{2^n}}$ and so it holds in some groups $D_{2^np}$ with $p$ an odd prime; however such groups are not nilpotent.  Therefore (a) and the first assertion of (b) follow.   
Corollary \ref{nilppi2} gives one implication in the proof of the second statement of (b).  To prove the other implication we consider wreath products. 

Suppose that $\varpi$ is an infinite set of primes.  Fix $q\in\varpi$, and let $(p_n)$ be a sequence in $\varpi\smallsetminus\{q\}$ tending to infinity. By $A\wr B$, we mean the standard wreath product $A^B\rtimes B$. Consider the groups $G_n=C_{q^n}\wr C_q$, and $H_n=C_{p_nq^n}\wr C_q$. Clearly taking ultraproducts commutes with taking wreath products with a given finite group (on the right). Hence, for every ultrafilter $\CU$, the ultraproducts $\prod_\CU G_n$ and $\prod_\CU H_n$ are isomorphic to respectively
\[\Big(\prod_\CU C_{q^n}\Big)\wr C_q\quad\text{and}\quad
\Big(\big(\prod_\CU C_{p_n}\big)\times \big(\prod_\CU C_{q^n}\big)\Big)\wr C_q.\] 
Since by Lemma \ref{summand_R}, $\prod_\CU C_{p_n}$ is isomorphic to $\QQ^{(\cont)}$ and $\prod_\CU C_{q^n}$ is isomorphic to $\big(\prod_\CU C_{q^n}\big)\times\QQ^{(\cont)}$, we deduce that the ultraproducts are isomorphic.  Therefore $(G_n)$ and $(H_n)$ are AEE sequences by Proposition \ref{AEE_ultra}.  However 
each $G_n$ is a $q$-group and hence nilpotent, and no $H_n$ is nilpotent. By Proposition \ref{AEE_car}, we conclude that no first-order sentence characterizes nilpotent groups among finite $\varpi$-groups.

If $f(x)$ were a first-order formula defining the Fitting subgroup, then the first-order sentence $(\forall x)f(x)$ would hold for all dihedral $2$-groups, and so also for all ultraproducts of such groups.  Hence from above it would hold, for every large enough prime $p$, for the non-nilpotent dihedral group $D_{2^np}$.
\end{proof}

\begin{proof}[Proof of Theorem \ref{psfn_subgroups}]
For every abelian group $A$, we write $\Dih(A)$ for the semidirect product $A\rtimes_\pm C_2$;  the sign $\pm$ indicates that the action on $A$ of the non-trivial element of $C_2$ is by inversion. In particular, $\Dih(C_n)$ is the dihedral group of order $2n$.  It is clear that for all $A$ the derived subgroup of $\Dih(A)$ is $A^2$ and it consists of commutators. 

Let $(p_n)$ be a sequence of primes tending to infinity and for each $n$ set $G_n=\Dih(C_{2^n})\times\Dih(C_{p_n})$.   Write
$\Dih(C_{2^n})=C_{2^n}\rtimes_{\pm}\langle a_n\rangle$ and
$\Dih(C_{p_n})=C_{p_n}\rtimes_{\pm}\langle b_n\rangle$. Fix a
non-principal ultraproduct $G$ of the sequence $(G_n)$ and let $a$, $b$ be the
elements of $G$ corresponding to the sequences $(a_n), (b_n)$.

Define $L=[G,G]\cup ab[G,G]$. Since $[G,G]$ consists of commutators, $L$
is definable.  It is also a subgroup, of index $4$ in $G$, and is isomorphic to the ultraproduct of the
groups $(C_{2^{n-1}})\times C_{p_n})\rtimes_\pm \langle a_nb_n\rangle$.  In
particular, it follows from the proof of Theorem
\ref{thm_abel}\ref{item_nil} that the latter is isomorphic to the
ultraproduct of the groups $\Dih(C_{2^n})$ (over the same non-principal ultrafilter).
Hence $L$ is pseudo-(finite nilpotent).

Define $H=L\cap \C_G(a)$; it is thus definable. Note that $[G,G]\cap \C_G(a)$ is
the ultraproduct of the groups $C_{p^n}$, and $H$ is the semidirect product
$([G,G]\cap \C_G(a))\rtimes_\pm\langle ab\rangle$. In particular, the second
projection, restricted to $H$, is an isomorphism. Therefore, $H$ is
isomorphic to the ultraproduct of the groups $\Dih(C_{p^n})$ and hence has
trivial centre. Since $H$ has trivial centre and is not trivial, it is not
pseudonilpotent.
\end{proof}

\begin{rem}\label{nilpotent_inside}
Theorem \ref{psfn_subgroups} shows that inside a pseudofinite group, the class of definable pseudo-(finite nilpotent) subgroups can be badly behaved. The following notion purports to remedy these shortcomings.  We say that a definable subgroup $H$ of a pseudofinite group $G$ is pseudonilpotent {\it inside} $G$ if, in the theory of pairs of a group with a subgroup, every formula satisfied by a pair consisting of a finite group and a nilpotent subgroup is satisfied by $(G,H)$. This clearly implies that $H$ is pseudo-(finite nilpotent). Moreover if subgroups  $H\leq L$ are definable in $G$ and $L$ is pseudonilpotent inside $G$, then so is $H$. Hence, in the given example where $H$ is not pseudo-(finite nilpotent), the pseudo-(finite nilpotent) definable subgroup $G$ is not pseudonilpotent inside $G$.
\end{rem}

\section{Perfect groups and Theorem \ref{thm_perfect}}
We use a construction similar to the one in \cite[Lemma 2.1.10]{holt}.

Let $n$ be a positive integer and $q=p^e\geq 3$ an odd prime power.
Let $V_n$ be the direct sum of $n$ copies of the natural module for $\mathrm{SL}_2(q)$ over $\FF_{q}$.  
By checking weights, one sees that the exterior square $W_n=\bigwedge^2V_n$, regarded as an $\mathrm{SL}_2(q)$-module, is a direct sum $Y_n\oplus Z_n$, where $Y_n$ is a direct sum of $n(n-1)/2$ copies of the irreducible $3$-dimensional $\mathrm{SL}_2(q)$-module and $Z_n$ is a module of dimension $n(n+1)/2$ on which $\mathrm{SL}_2(q)$ acts trivially.  

Define a multiplication operation on the set $E_n\colon=V_n\times W_n$ by $(v_1,w_1)(v_2,w_2) = (v_1+v_2, w_1+w_2 + v_1\wedge v_2)$. It is easy to check that $E_n$ becomes a group of exponent $p$ and that $[(v_1,w_1),(v_2,w_2)] = (0, v_1\wedge v_2)^2$. Thus $E_n$ is nilpotent of class $2$ with commutator subgroup and centre isomorphic to $W_n$.  

For $\theta\in \mathrm{SL}_2(q)$, write $\theta$ also for the automorphism induced by $\theta$ in $W_n$; then the map   
$(v,w) \mapsto (\theta v,\theta w)$ is easily checked to be an automorphism of $E_n$.  In this way we can define an action of $\mathrm{SL}_2(q)$ on $E_n$.
The image of $Y_n$ in $E_n$ is normal and $\mathrm{SL}_2(q)$-invariant.  Define $G_n\colon=E_n/Y_n$  and identify $Z_n$ with its image in $G_n$. Thus
$Z_n$ is central in $G_n$ and acted on trivially by $\mathrm{SL}_2(q)$, and we have $G_n/Z_n\cong V_n$.   Moreover there is an isomorphism from a vector space of dimension  $\frac{n}2(n+1)$ over $\FF_q$ to $Z_n$; we call the image of a $1$-dimensional space a {\em line} in $G'_n$. 

\begin{lem}\label{comlength}
For every $k\geq 1$ and $n\geq 8k+2$ and every prime power $q$, the group $G_n'=G_n(q)'$ contains a line in which no non-trivial element is a product of $k$ commutators in $L=G_n(q)\rtimes\SL_2(q)$.
\end{lem}
\begin{proof}
The quotient group $L/G_n'$ has order $q^{2n}(q^3-q)\leq q^{2n+3}$. Since $G_n'$ is central, the commutator map $L\times L\to L$ factors through a map $L/G_n'\times L/G_n'\to L$. Hence $L$ has at most $q^{2(2n+3)}$ commutators and at most $q^{2k(2n+3)}$ products of $k$ commutators.

The group $G_n'$ consists of $(q^{n(n+1)/2}-1)/(q-1)> q^{n(n+1)/2-1}$ lines, pairwise intersecting in $1$. Hence there is a line containing no non-trivial product of $k$ commutators as soon as $q^{2k(2n+3)}\leq q^{n(n+1)/2-1}$, that is, as soon as $2k(2n+3)\leq \frac12 n(n-1)$. For $n\geq 8k+2$ we have 
\[2k(2n+3)\leq \frac{(n-2)(2n-3)}4=\frac{2n^2-n-6}4\leq \frac{2n^2-2n}{4}=\frac{n(n-1)}2.\qedhere\]
\end{proof}

From now on, assume that $q\equiv \pm1 \pmod {10}$. This is, by Dickson's classification \cite{Dickson}, the condition that ensures that $\mathrm{SL}_2(q)$ contains a copy of the binary icosahedral group $B$, a perfect group of order $120$. Therefore we have actions of $B$ on $\FF_{p}^2$, $V_n$ and $G_n$; since the action on $\FF_{q}^2$ is clearly irreducible, $V_n$ is a sum of irreducible $\FF_{p}B$-modules of dimension $2$.  As a consequence, the semidirect product $V_n\rtimes B$ is a perfect group.  Since $Z_n$ is (contained in) the derived subgroup of $G_n$, we deduce that $H_n\colon=G_n\rtimes B$ is also a perfect group, in which $Z_n$ is central.   

\begin{lem}\label{summand}
Let $F$ be a finite group. Let $N$ be a torsion-free divisible nilpotent group on which $F$ acts, let $Z$ be a divisible central subgroup that has trivial intersection with the derived subgroup $N'$ of $N$, and suppose that $F$ acts trivially on $Z$. Then $N=Z\times M$ for an $F$-invariant subgroup $M$. 
\end{lem}
\begin{proof}
For $N$ abelian this follows from Maschke's theorem applied to $N$ as a $\QQ F$-module.  In general, the quotient $N/N'$ is torsion-free 
from \cite[Theorem 8.13]{Warf} and so from the abelian case we have $N/N'=N'Z\times M/N'$ for an $F$-invariant subgroup $M$ containing $N'$.  Thus $N=ZM$ and $Z\cap M\leq Z\cap N'=1$.   Clearly $M\triangleleft\, N$ and so $N=Z\times M$.
\end{proof}

\begin{prop}
Let $q_n=p_n^{e_n}$ be a sequence of prime powers with $q_n\equiv \pm 1\pmod {10}$, such that $p_n\to\infty$.  For each $n\geq0$ let $G_n(q_n)$, $H_n(q_n)$ be the groups constructed as above for the prime power $q_n$.
Then every non-principal ultraproduct of the sequence $(H_n(q_n))$ admits ${\QQ}^{(\cont)}$ as a direct summand. 
\end{prop}
\begin{proof}

Let $c_n$ be the largest integer such that $G(q_n)'$ has a line
$L_n$ containing no non-trivial product of $c_n$ commutators in $H(q_n)$. By Lemma \ref{comlength} we have $c_n\to\infty$.

Fix a non-principal ultrafilter $\CU$ on $\N$, and let $H=\prod_\CU H_n(q_n)$; thus $H=G\rtimes B$ with $G=\prod_\CU G_n(q_n)$. 
The group $L=\prod_\CU L_n$ has cardinality $\cont$ and is abelian, torsion-free and divisible, hence isomorphic to $\QQ^{(\cont)}$. Also $L$ has trivial intersection with the derived subgroup of $G$, and is central in $H$.  By Lemma \ref{summand}, we can write $G=L\times N$ for a $B$-invariant subgroup $N$ and we have $H=(N\rtimes B)\times L$.
\end{proof}

\begin{cor}\label{perfectAEE}
Let $(q_n)$ be a sequence of prime powers as above and let $(p'_n)$ be another sequence of primes with $p'_n\to\infty$. Then the sequence $(H_n(q_n))$ of perfect groups constructed above and the sequence $(H_n(q_n)\times C_{p'_n})$ of non-perfect groups are AEE.
$\qed$
\end{cor}

Theorem B follows immediately from Corollary \ref{perfectAEE}.

\begin{rem}
{\rm Let $p$ be a fixed prime, and let $q_n=p^{e_n}$ with $q_n\to\infty$.  An easy variant of the proof shows that every non-principal ultraproduct of the sequence $(H_n(p^{e_n}))$ admits ${\FF_p}^{(\cont)}$ as a direct summand.  In this case, we deduce that for every sequence $(k_n)$ tending to infinity, the sequences $(H_n(p^{e_n}))$ and $(H_n(p^{e_n})\times C_p^{k_n})$ are AEE.}
\end{rem}

\section{Products and powers of simple groups}

We start with the following positive result.

\begin{prop}\label{form_produ}
There exists a first-order sentence $K$ such that a finite group satisfies $K$ if and only if it is a finite direct product of non-abelian simple groups.
\end{prop}
\begin{proof}Since this is an immediate variation on results of \cite{Wcomp}, we only sketch the proof. The basic underlying fact is that a finite group is a finite direct product of non-abelian finite simple groups if and only if it satisfies the three conditions: (a) every quasisimple subnormal subgroup is normal, (b) every element is a commutator, (c) no nontrivial element commutes with all its conjugates. The verification is easy, using Schreier's and Ore's conjectures (which are both theorems).

Clearly (b) and (c) are characterized by a single first-order sentence. By \cite[Theorem 1.3]{Wcomp}, there exist first-order formulas $\pi(h,x)$ and $\pi'_{\mathrm{c}}(h)$ such that for every finite group $G$, the quasisimple subnormal subgroups of $G$ are the subsets $\pi(h,G)$ where $h$ ranges over $\pi'_{\mathrm{c}}(G)$. Hence (a) can be characterized by the sentence \[(\forall h)(\forall x)(\forall y)\Big(\big(\pi'_{\mathrm{c}}(h)\wedge \pi(h,x)\big)\to \pi(h,yxy^{-1})\Big).\qedhere\]
\end{proof}


Let us now prove Theorem \ref{thm_power}, which asserts that, among finite groups, the direct powers (or direct squares) of non-abelian simple groups cannot be characterized by a first-order sentence.
This follows from the following result:

\begin{lem}\label{lem_nplusone}
Let $(G_n)$ be a sequence of finite groups that elementarily converges. Then $(G_p\times G_q)$ elementarily converges when $\max(p,q)$ tends to infinity. In particular, $(G_n^2)$ and $(G_n\times G_{n+1})$ are asymptotically elementarily equivalent.
\end{lem}  
\begin{proof}
We remark that the proof below without essential changes works for other structures than groups.

We start by proving the lemma under the continuum hypothesis (CH), which implies that all non-principal ultraproducts of the sequence $(G_n)$ are isomorphic (see Remark \ref{CH_rem}) to a single group $G$.

Let $\CU$ be an non-principal ultrafilter on $\N$, and $\CU'$ its push-forward by $n\mapsto n-1$. Then
\[\prod_\CU G_n^2\cong \prod_\CU G_n\times\prod_\CU G_n\cong G\times G,\]
while 
\[\prod_\CU G_n\times G_{n+1}\cong \prod_\CU G_n\times\prod_\CU G_{n+1}\cong \prod_\CU G_n\times\prod_{\CU'} G_n\cong G\times G.\]
Hence the conclusion follows from Proposition \ref{AEE_ultra}: the assertion of the lemma is a theorem of ZFC+CH.

The next step is to eliminate CH. We use Shoenfield's absoluteness theorem \cite{Sho}. Its content is that a sufficiently simple assertion, if true in ZFC+CH, is also true in ZF. 

Let us be more precise. Let $V$ be a model of ZF. Then inside $V$ there is submodel $L$ of ZF, called G\"odel's constructible universe, which is a model of ZFC+CH. Shoenfield's absoluteness theorem states (in particular) that any sentence of set theory of the form $(\forall I\subset\N)(\exists J\subset\N)\Psi(I,J)$, where $\Psi(I,J)$ only involves quantifiers over $\N$ and no parameters, is absolute in the sense that
it holds in $V$ if and only if it holds in $L$.  We are concerned with a property of
the simpler form $(\forall I\subset\N)\Phi(I)$, with $\Phi(I)$ parameter-free and with only quantifiers over $\N$.  From above we know that this formula holds in every model of ZFC+CH. Hence, it also holds in $L$, and thus in the original model $V$, by Shoenfield's absoluteness theorem.

Let us now see why the assertion of the theorem can be written in the prescribed form $(\forall I\subset\N)\Phi(I)$.

It is convenient first to restate the theorem as follows: start with an explicit enumeration $(H_n)$ of all finite groups (so that each finite group is isomorphic to $H_n$ for some $n$). This allows avoiding quantifiers on sequences of groups. The assertion of the theorem can be stated as follows: for every subset $I$ of $\N$ such that $(H_n)$ elementarily converges on $I$ (when $n\to\infty$), the sequence $(H_p\times H_q)$ elementarily converges when $p,q\to\infty$, $p,q\in I$. Letting $\mathcal{F}$ be the (countable) set of first-order sentences of group theory we can write this as
\begin{align*}
(\forall I\subset\N)&
 \Big[\Big((\forall F\in\mathcal{F})(\exists n\in\N)(\forall p,q\in\N)\big(p,q\in I_{\geq n} \to \big(H_p\models F\leftrightarrow H_q\models F)\big)\Big)\\
 \to &\Big( (\forall F\in\mathcal{F})(\exists n\in\N)(\forall p,p',q,q'\in\N)\\
 &\qquad\big(p,p',q,q'\in I_{\geq n} \to \big(H_p\times H_q\models F\leftrightarrow H_{p'}\times H_{q'}\models F)\big)\Big)\Big].
\end{align*}
Here ``$p\in I_n$" is shorthand for $(p\in I \wedge p\geq n)$. This has the required form $(\forall I\subset\N)\Phi(I)$, with $\Phi(I)$ parameter-free and having only quantifiers over $\N$. Hence, Shoenfield's absoluteness theorem applies.
\end{proof}


\begin{proof}[Proof of Theorem \ref{thm_power}]
Start from any infinite set of isomorphism classes of non-abelian finite simple groups. From it, extract a subsequence $(S_n)$ that elementarily converges, with $|S_n|<|S_{n+1}|$ for all $n$. By Lemma \ref{lem_nplusone}, 
the sequences $(S_n^2)$ and $(S_n\times S_{n+1})$ are AEE.  The result now follows
 by Proposition \ref{AEE_car}.
\end{proof}

\begin{rem}
Theorem \ref{thm_power} is based on the existence of an elementarily convergent sequence of non-abelian finite simple groups of order tending to infinity, whose existence follows from a compactness argument. While it seems hard to characterize such sequences fully, one can exhibit some of them, relying on the work of Ax \cite{Ax} on pseudofinite fields. For instance, his results imply that for each given prime $p$ and integer $d\geq 2$, the sequence $\big(\mathrm{PSL}_d(\mathbb{F}_{p^{n!}})\big)_{n\geq 1}$ is elementarily convergent.
\end{rem}


\section{Definable supplements: Theorem \ref{thm_defina}}
To prove Theorem \ref{thm_defina} it will suffice to prove the following result.  

\begin{prop}\label{frattini_prop}  Let $G$ be a pseudofinite group and $K_1$ a definable normal subgroup not contained in $\R(G)$.  Then $K_1$ has a proper definable supplement in $G$.

More precisely, let $\varphi(y_1,\dots,y_n,x)$ be a first-order formula and $\underline a=(a_1,\dots,a_n)$ an $n$-tuple of elements with $K=\varphi(\underline a,G)$ a normal subgroup of $G$.  
Then there exist a first-order formula $$\chi(y_1,\dots, y_n, w_1,w_2,z,x),$$ depending only on $\varphi$, and elements $h_1,h_2, k$ of $G$, such that $\chi(\underline a, h_1,h_2,k, G)$ is a proper supplement to $K_1$ in $G$.
\end{prop}

First we establish most of the necessary group-theoretic statements:

\begin{lem}\label{lem_frattini}
Let $\bar{F}$ be a finite group, $\bar{L}=\prod_{i=1}^r\bar{S}_i$ a non-abelian minimal normal subgroup, with $\bar{S}_i$ simple. Let $\bar{T}$ be the kernel of the action of $\bar{F}$ by conjugation on the set $\{\bar{S}_i\mid 1\leq i\leq r\}$ of normal subgroups of $\bar{L}$. Let $\bar{P}=\prod_i\bar{P}_i$ be a Sylow subgroup of $\bar{L}$, with $\bar{P}_i=\bar{P}\cap\bar{S}_i$. Let $\bar{D}$ be the centralizer of $\bar{P}$ in $\bar{F}$. Then
\begin{enumerate}[label={\rm (\alph*)}]
\item\label{item_fratcen} $\bar{F}=\bar{L}\NN_{\bar{F}}(\bar{D})$;
\item\label{item_notn} if $\bar{P}\neq\{1\}$ then $\NN_{\bar{F}}(\bar{D})$ is a proper subgroup of $\bar{F}$;
\item\label{item_Dcentr} if $\bar{P}_i=\langle \bar{s}_i\rangle$ is non-trivial cyclic and $\bar{s}=\bar{s}_1\dots \bar{s}_r$, then $\bar{D}$ is the centralizer of $\bar{s}$ in $\bar{T}$;
\item\label{item_d1d2} there exist $\bar{d}_1,\bar{d}_2\in\bar{F}$ $($in fact in $\bar{S}_1)$ such that $\bar{T}=\bigcap_{g\in\bar{F}}g^{-1}Q(\bar{d}_1,\bar{d}_2)g$, where $Q(\bar{d}_1,\bar{d}_2)=\bar{F}\smallsetminus\bigcap_{1\leq i,j\leq 2}\{x\mid [\bar{d}_i,\bar{d}_j^x]=1\}$.
\end{enumerate}
\end{lem}
\begin{proof}
By the Frattini argument, we have $\bar{F}=\bar{L}\NN_{\bar{F}}(\bar{P})$. Since $\NN_{\bar{F}}(\bar{P})\leq \NN_{\bar{F}}(\bar{D})$, we deduce \ref{item_fratcen}.

Clearly if $\bar{P}\neq\{1\}$ then $\bar{D}\cap\bar{S}_1$ is the normalizer in $\bar{S}_1$ of the centralizer of $\bar{P}_1$ in $\bar{S}_1$, and is not normal. So $\bar{D}$ is not normal, proving \ref{item_notn}.

\ref{item_Dcentr} That $\bar{P}\neq\{1\}$ ensures that the centralizer $\bar{D}$ of $\bar{P}$ is contained in $\bar{T}$, hence contained in $C_{\bar{T}}(\bar{s})$. Conversely if $x\in\bar{T}$ centralizes $\bar{s}$ then since its action by conjugation preserves the product decomposition of $\bar{L}$, the element $x$ centralizes each $\bar{s}_i$, hence $x$ centralizes $\bar{P}$.

\ref{item_d1d2}  Every finite simple group can be generated by two elements, and hence, if non-abelian, has a pair with trivial centralizer. Choose in $\bar{S}_1$ such a pair $(\bar{d}_1,\bar{d}_2)$. Then the normalizer of $\bar{S}_1$ in $\bar{F}$ is $\big\{x\in\bar{F}\mid[\bar{S}_1,x\bar{S}_1x^{-1}]\neq\{1\}\big\}$; this is precisely $Q(\bar{d}_1,\bar{d}_2)$. Hence the intersection of its conjugates is $\bar{T}$.
\end{proof}

\begin{proof}[Proof of Proposition \ref{frattini_prop}]
We write $\underline y=(y_1,\dots,y_n)$. Let $\psi(\underline y)$ be the first-order formula
\[\varphi(\underline y , 1)\wedge(\forall u\forall x_1\forall x_2)\Big(\big(\varphi(\underline y ,x_1)\wedge\varphi (\underline y ,x_2)\big)\to \varphi(\underline y ,ux_1x_2^{-1}u^{-1})\Big).\]
This is just the statement that $\varphi(\underline y,G)$ is a normal subgroup. Write $\psi'(\underline y)$ for the sentence $\psi(\underline y)\wedge(\exists x)(\varphi(\underline y,x)\wedge \neg\rho(x))$. For pseudofinite groups, this says that the definable subset $\varphi(\underline y,G)$ is a normal subgroup that is not 
pseudo-(finite soluble).

Now let $F$ be a finite group, $\underline b=(b_1,\dots,b_n)$ an $n$-tuple with $\psi'(\underline b)$ holding in $F$
and $K=\varphi(\underline b,F)$; thus $K\triangleleft F$ and $K\not\leq \R(F)$. Define $M=K\cap\R(F)$, and let $L$ be a normal subgroup, minimal among those normal subgroups of $F$ contained in $K$ and properly containing $M$.
Thus, $L/M$ is a non-abelian minimal normal subgroup of $F/M$ and so a direct product of isomorphic non-abelian simple groups $\overline S_1,\dots, \overline S_r$.  From \cite[Theorem 4.9]{KLST} every
finite simple group has a non-trivial cyclic Sylow subgroup, and so we can choose a prime $p$ for which each $\bar{S}_i$ has a non-trivial cyclic Sylow $p$-subgroup,
generated by $Ms_i$ say, with $s_i$ in $L$. So $M\langle s_1,\dots,s_r\rangle/M$ is a $p$-Sylow subgroup of $L$; let $D$ be the inverse image in $F$ of its centralizer in $F/M$. By Lemma \ref{lem_frattini}, \ref{item_fratcen} and \ref{item_notn}, $\NN_F(D)$ is a proper supplement to $L$, and hence to $K$.

Write $s=s_1\dots s_r$ and let $T$ be the kernel of the conjugation action of $F$ on the set $\{\overline S_i\mid 1\leq i\leq r\}$ of subgroups of $F/M$. By Lemma \ref{lem_frattini}\ref{item_Dcentr}, $D/M$ is the centralizer of $Ms$ in $T/M$; that is, $D=\{x\in F\mid\rho([s,x])\}\cap T$.

By Lemma \ref{lem_frattini}\ref{item_d1d2}, there exist $d_1,d_2\in F$ (in the inverse image of $\bar{S}_1)$ such that $T=\theta_1(d_1,d_2,F)$, where, $\equiv$ denoting equality of formulae:
\[\theta_0(w_1,w_2,x)\equiv\bigvee_{1\leq i,j\leq 2}\!\!\!\neg \rho([w_i,w_j^x]),\quad \theta_1(w_1,w_2,x)\equiv(\forall g)\big(\theta_0(w_1,w_2,gxg^{-1})\big).\]
 
 Hence $D$ equals $\theta(d_1,d_2,s,G)$ where $\theta(w_1,w_2,z,x)\equiv\theta_1(w_1,w_2,x)\wedge \rho([z,x])$.  Therefore $\NN_F(D)=\chi(d_1,d_2,s,G)$ where 
 $$\chi(w_1,w_2,z,x)
\equiv(\forall y)\big(\theta (w_1,w_2,z,y)\to\theta(w_1,w_s,z,y^x)\big);$$
 this a subset of elements that in every finite group is guaranteed to be a subgroup for any choice of the triple of parameters, and hence also in every pseudofinite group. The assertion that $\chi(d_1,d_2,s,G)$ is a proper supplement to $\varphi(\underline y,G)$ can be encapsulated in the first-order sentence $\xi(w_1,w_2,x,\underline y)$:
\[(\exists t)(\neg\chi(w_1,w_2,x,t))\wedge (\forall u)(\exists v_1\exists v_2)\big(\chi(w_1,w_2,x,u)\wedge\varphi(\underline y,u)\big).\]
 
 Thus our finite group $F$ together with its $n$-tuple $\underline b$ satisfies $(\exists w_1\exists w_2\exists x)(\xi(w_1,w_2,x,\underline b))$. Hence, we have shown that every finite group $F$ satisfies the first-order sentence 
 \[(\forall\underline y)\Big(\psi'(\underline y)\to\Big((\exists w_1\exists w_2\exists x)\big(\xi(w_1,w_2,x,\underline y)\big) \Big)\Big).\]
 Therefore, this sentence holds for the pseudofinite group $G$. Since $\psi'(\underline a)$ holds in $G$, we deduce that there exist elements $d_1,d_2,s$ in $G$ such that $\xi(d_1,d_2,s,\underline a)$ holds in $G$. So $\chi(d_1,d_2,s,G)$ is a proper supplement to $\varphi(\underline a,G)$, and hence to $K$.  
\end{proof}


\begin{thebibliography}{99}
\bibitem{Ax} J. Ax. The elementary theory of finite fields. {\em Ann. Math.} (2) {\bf88} (1968),
239--271.


\bibitem{Dickson} L.E.\ Dickson. {\em Linear Groups, with an Exposition of the Galois Field Theory}.
(Teubner, Leipzig, 1901).

\bibitem{eklof}  P.C.\ Eklof. Ultraproducts for Algebraists.   In {\em Handbook of Mathematical Logic} (North-Holland Publishing Company, 1977), pp.\ 105--137.

\bibitem{felgner}  U.\ Felgner.  Pseudo-endliche Gruppen.   In
{\em Proceedings  of the 8th Easter Conference on Model Theory} 
(Humboldt--Universit\"at, Berlin,1990), pp.\ 82--96.

\bibitem{Frattini} G.\ Frattini. Intorno alla generazione dei gruppi di operazioni.  {\em Rend.\ Atti Acad.\ Lincei} 1: 4 (1885) 281--285; 455--457.

\bibitem{holt} D.F.\ Holt and W.\ Plesken. {\em Perfect Groups.}  Oxford Mathematical Monographs (The Clarendon Press, Oxford, 1989).

\bibitem{KLST}  W.\  Kimmerle,  R.\  Lyons,  R.\  Sandling and  D.N.\  Teague.  Composition factors from the group ring and Artin's theorem on orders of simple groups.   {\em Proc London Math.\ Soc.}  (3)  {\bf60}  (1990),  89--122.

\bibitem{Macp} D.\ Macpherson.  Model theory of finite and pseudofinite groups.  {\em Arch. Math. Logic} {\bf57} (2018), 159--184.

\bibitem{OHP}  A.\ Ould-Houcine and F.\ Point.  Alternatives for pseudofinite groups.  {\em J. Group Theory} {\bf16} (2013), 461--495.

\bibitem{Poi} B.\ Poizat. {\em A Course in Model Theory.}  Universitext (Springer-Verlag, New York, 2000).

\bibitem{Sho} J. Shoenfield. The problem of predicativity. In {\em 1961 Essays on the Foundations of Fathematics}  (Magnes Press, Hebrew Univ., Jerusalem), pp.\ 132--139.

\bibitem{Szm} W.\ Szmielew.
Elementary properties of Abelian groups.
{\em Fund.\ Math.} {\bf41} (1955), 203--271. 

\bibitem{Warf} R.B.\ Warfield.  {\em Nilpotent Groups.}  Lecture Notes in Math. vol. 513 (Springer-Verlag, Berlin--New York, 1976).


\bibitem{Wsol} J.S.\ Wilson.  Finite axiomatization of finite soluble groups.  {\em J.\ London Math.\ Soc.} (2) {\bf74}
(2006), 566--582.

\bibitem{Wrad} J.S.\ Wilson.  First-order characterization of the radical of a finite group. {\em J. Symbolic Logic} {\bf 74} (2009), 1429--1435.

\bibitem{Wcomp} J.S.\ Wilson.  Components and minimal normal subgroups of finite and pseudofinite groups.  {\em J. Symbolic Logic} {\bf84} (2019), 290--300.
\end{thebibliography}
\end{document}